\documentclass[10 pt, twocolumn]{article}
\usepackage{amsmath, amssymb,amsfonts, amsthm,caption, subcaption}
\usepackage{graphicx}
\usepackage{array}
\usepackage{diagbox}
\usepackage{color}

\makeatletter
\let\@fnsymbol\@arabic
\makeatother

\newcommand{\R}{\mathbb{R}}

\newcommand{\Z}{\mathbb{Z}}

\newcommand{\Mod}[1]{\ (\text{mod}\ #1)}

\newtheorem{thm}{Theorem}
\newtheorem{definition}{Definition}
\newtheorem{lemma}{Lemma}

\newtheorem{conjecture}{Conjecture}
\newtheorem{remark}{Remark}

\DeclareMathOperator{\sgn}{\mathrm{sgn}}

\title{On a very steep version of the standard map.}
\author{M. Arnold\thanks{University of Texas at Dallas, Richardson, TX, 
USA}\and T. 
Dauer\thanks{Indiana University, Bloomington, IN, USA} \and M. 
Doucette\thanks{University of Chicago, Chicago, IL, USA} \and S.-C. 
Wolf\thanks{Universit\'{e} Paris-Saclay, Paris, France}}

\begin{document}

	\maketitle

		We consider the long time behavior of the trajectories of the 
		discontinuous 
		analog of the standard Chirikov map. We prove that for some values of 
		parameters all the trajectories remains bounded for all time. For other 
		set of parameters we provide an estimate for the escape rate for the 
		trajectories and present a numerically supported 
		conjecture for the 
		actual escape rate.

	\section{Introduction.}
	
We consider the area-preserving transformation of the cylinder 
	$[0,1) \times \R $ defined by $f(x,y)=(x',y')$ where
	
	\begin{equation}
	\begin{cases}
		\label{eq:original map}
		x' &= x + \alpha y \Mod{1} \\
		y' &= y + \sgn \left(x'-\frac{1}{2}\right), \\
\end{cases}
	\end{equation}
	Parameter $\alpha \in \R$ is called the twist parameter. A point 
	at position $(x,y)$ on the cylinder moves at constant height $y$ around 
	the cylinder a distance $\alpha y$, and then moves up one unit if it is on 
	the right half of the cylinder ($x' \in (1/2,1)$), down one unit if it is on 
	the left half ($x' \in (0,1/2)$), and stays at the same vertical position if it 
	is at the singular lines $x'=1/2$ or $0$.
	
	Such system can be regarded as a discontinuous analog of the standard 
	Chirikov map (see \cite{CHIRIKOV1979}), where the smooth function 
	$\sin(x')$ 
	is replaced by the discontinuous $\sgn(x')$. This system can be also 
	obtained from the Fermi-Ulam accelerator model with the sawtooth-like 
	wall movement regime (see \cite{arnold2015} for details).  For the 
	smooth variants of the 
	described problems KAM-technique can be used to provide the existence 
	of the invariant curves separating the phase space and so no unbounded 
	orbit exists for such systems. Since  transformation \eqref{eq:original 
	map} 
	is  discontinuous, KAM theory is not applicable and so new methods are 
	needed for the analysis. Such systems having many interesting 
	dynamical properties, attracted a lot of attention in the past few years 
	(see e.g. \cite{icmp_dima}, \cite{DolSim}). 
	
	In this note we study the asymptotic properties of the orbits of system 
	\eqref{eq:original map} in terms of the growth rate of the height $y_n$ of 
	the iterates $(x_n,y_n)=f^n(x_0,y_0)$. We will focus on the rational 
	values of the twist parameter $\alpha$. The case of the irrational values 
	of $\alpha$ is 
	more difficult and will be a subject of a future work.

	In the next section, we collect preliminary results on the structure of the 
	set of orbits of system \eqref{eq:original map} and relate our system to  
	a transformation on a finite lattice. In section \ref{sec:Main} we present 
	our main results and state some conjectures 
	based on the numerical simulations. Section \ref{sec:periodic} is devoted 
	to the numerical study of the periodic orbits. 
	
	\paragraph{Acknowledgments.} Present work was done during the  
	Summer@ICERM research program in 2015. Authors are deeply thankful 
  to ICERM 
	and Brown University for the hospitality and highly encouraging 
	atmosphere. Authors also want to thank 
	Vadim Zharnitsky and Stefan Klajbor-Goderich for deep and fruitful 
	discussions.

	\section{Preliminaries.}
	We will use the following notations. 
	Integer part of $x$ is denoted as $\lfloor x\rfloor$, therefore for the 
	fractional part of $x$ we have $\{x\}=x-\lfloor x\rfloor$.  
	$\Z_q=\{0,1,\dots,q-1\}$ will denote the ring of residues modulo $q$.
	\begin{lemma}
		\label{lm:symmetry}
		 Map \eqref{eq:original map} is symmetric with respect to 
		the point 
		$(1/2,0)$. In details, for any 
		two points $(x,y)$ and $(\tilde{x},\tilde{y})$ such that 
		$(x,y)+(\tilde{x},\tilde{y})=(1,0)$ one has 
		$f(x,y)+f(\tilde{x},\tilde{y})=(1,0)$ (See Fig. \ref{fig: IET}).
	\end{lemma}
	\begin{proof}
	  Let $(\tilde{x},\tilde{y}) = (1-x,-y)$
	  , then $x' = x+\alpha y $  and
	  $\tilde{x}'= 1-x-\alpha y=1-x'$. Hence $\tilde{y}' -\tilde{y}=-(y'-y)$.
\end{proof}

\begin{figure}[hbt]
	\centering
	\includegraphics[width=0.35\textwidth]{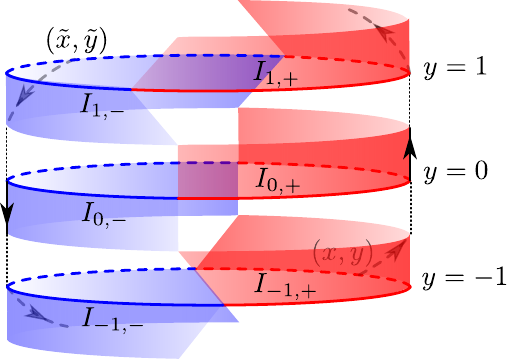}
	\caption{Transformation \eqref{eq:original map} can be 
	thought as 
	interval 
	exchange transformation with infinitely many intervals. Intervals $I_{j,+}$ 
	are the pre-images of the segments $x\in (1/2,1),\, y=y_0+(j+1)$. 
	Similarly, 
	$I_{j,-}$ are the pre-images of $x\in (0,1/2), y=y_0+(j-1)$. }\label{fig: 
	IET}
\end{figure}

Next we notice that the transformation \eqref{eq:original map} preserves 
the lattice $y\in \{y_0+\Z\}$ in the second coordinate. Let 
$\alpha=\dfrac{p}{q}$, where $p,q\in\Z$ and $\gcd(p,q)=1$. Then for two 
points $(x,y)$ and 
$(x,y\pm q)$ first 
components of their images coincide and thus the increments in the 
second 
components are equal. Therefore we can restrict our attention to the set 
$(x,y)\in [0,1)\times\{y_0+\Z_q\}$. Transformation \eqref{eq:original 
map} can be regarded as an interval exchange transformation on the union 
of $2q$ intervals $\bigcup_{j=1}^q (I_{j,+}\cup I_{j,-})$, where 
$I_{j,-}=\{(x,y): \{qx+p y<q/2\}, y=y_0+j\}$ and $I_{j,+}=\{(x,y): 
\{qx+p y>q/2\}, y=y_0+j\}$ (see Fig. \ref{fig: IET}) 

If one considers the special case $q=1$ dynamics of the 
system 
\eqref{eq:original map} degenerates to 
\begin{equation*}\label{eq:Ralston}
\begin{cases}
x' = x+ py_0\Mod 1 \\
j' = j + \sgn\left(x'-\frac{1}{2}\right)  
\end{cases}
\end{equation*} In this particular case dynamic of the 
first coordinate became independent from the second coordinate. 
Therefore one can identify all the intervals $I_{j,+}=I_+$ and all the 
intervals 
$I_{j,-}=I_-$.  
For  $py_0=1$ one immediately obtain linearly 
growing trajectory $f^n(1/4,y_0)=(1/4, n+y_0)$. On the other hand for 
$py_0=1/2$ any trajectory remains bounded since 
for $x\in I_+$ from lemma \ref{lm:symmetry} it follows that $x'\in I_-$. 
The case of $y_0$ 
being irrational has been extensively studied (see 
\cite{discrep, Kesten, Ralston,Avila}) It provides a 
random-like behavior of 
the trajectories depending on the arithmetic properties of the initial 
condition $y_0$.

In this paper we address the case $q>1$ and consider rational initial  
conditions $y_0=\dfrac{a}{b}$. Using substitution $y=y_0+j$, we rewrite 
transformation \eqref{eq:original map} as

\begin{equation}\label{eq:prebands}
\begin{cases}
x' = x+ \frac{p(a+bj)}{bq}\Mod 1 \\
j' = j + \sgn\left(x'-\frac{1}{2}\right)
\end{cases}
\end{equation} 
where $j'$ is defined by the expression $y'=y_0+j'$.

	\begin{lemma}
		\label{lm:torus}
		Trajectories of the system \eqref{eq:prebands} are organized in  
		bands: for $(x,j)$ and $(\tilde{x}, j)$ such that 
		$\lfloor bqx \rfloor=\lfloor bq\tilde{x}\rfloor$ it 
		follows that  $\lfloor bqx'\rfloor=\lfloor 
		bq\tilde{x}'\rfloor$. 
	\end{lemma}
	\begin{proof}
		Obviously, integral parts of $xbq$ and $\tilde{x}bq$ are changed by 
		the transformation \eqref{eq:prebands} by the same amount
		$p(a+bj)$. 
		\end{proof}

From lemma \ref{lm:torus} it follows that we can restrict our attention on 
the single representatives from the classes of equivalent trajectories and 
consider our transformation on the discrete torus $(x,y)\in \Z_{bq}\times 
\Z_q$. For the sake of simplicity we will use the following lattices:

	\begin{equation*} 
	L_e= \left\{ \left(\frac{2+4r}{4bq}, \frac{a}{b}+j\right),~ j \in 
	\Z_q,~ r \in \Z_{bq}\right\} \end{equation*}
	
	\[ L_o = \left\{ \left(\frac{3+4r}{4bq},\frac{a}{b}+j\right),~ j \in 
	\Z_q,~ r \in\Z_{bq} \right\} \]
	
We refer to $L_e$ and $L_o$ as the even and the odd lattice, 
	respectively. $L_o$ can be obtained from $L_e$ by shifting to the right 
	by $\dfrac{1}{4bq}$ (see Fig. \ref{fig: lattices}).
	
	\begin{figure}
	\centering 
	\includegraphics[width=0.35\textwidth]{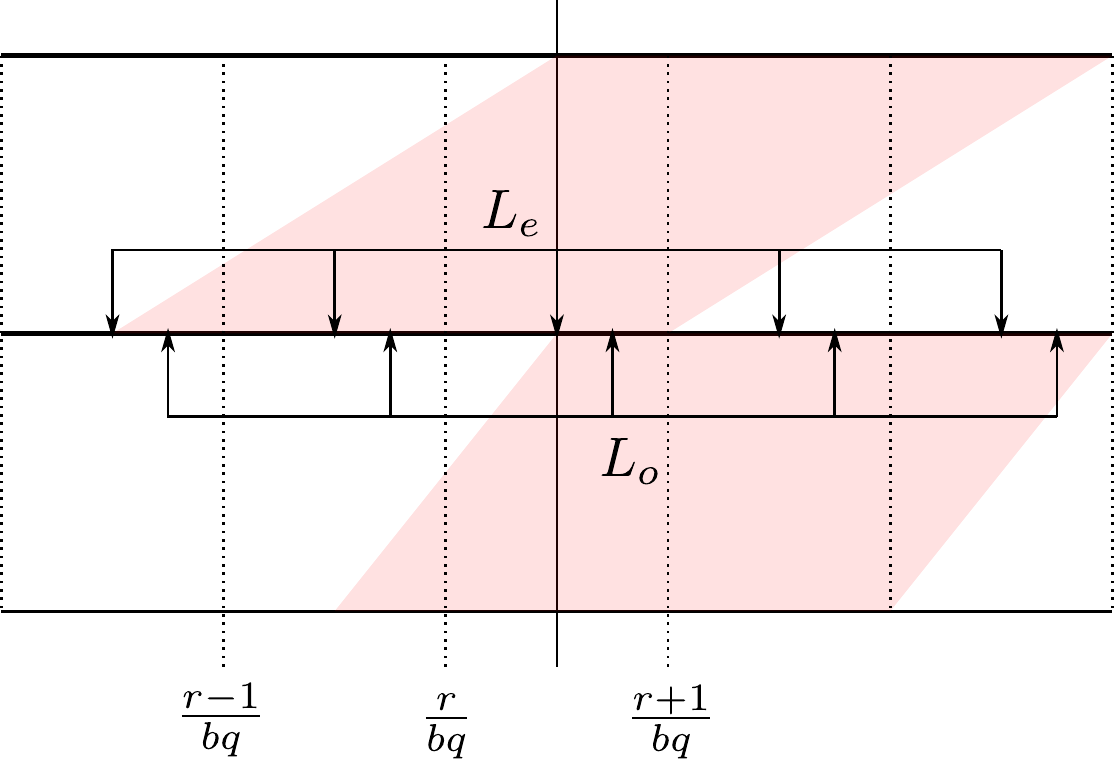}
	\caption{Lattices $L_e$ and $L_o$ for odd $bq$. }\label{fig: lattices}
	\end{figure}

Thanks to lemma \ref{lm:torus}  these lattices are invariant 
under the action of $f$. 
To simplify the notations, henceforth when $bq$ is even we consider $f:L_e 
\rightarrow L_e$, and when $bq$ is odd we consider $f: L_o \rightarrow 
L_o$. For purposes of calculation we will  think of $f$ as acting on $(r,j)$ 
instead of $(x,y)$. 
Explicitly, for $r \in \Z_{bq}$ and $j \in \Z_q$ we 
have
\begin{equation}
\label{eq:f with r}
\begin{cases}
r' = r + p(a+bj) &\Mod{bq} \\
j' = j + \sgn(2r'-bq+1+\delta)&\Mod {q}, 
\end{cases}
\end{equation}
where $\delta=bq\Mod{2}$ refers to our choice of the lattice 
$L_e$ or 
$L_o$.

	\section{Main Results. }
	\label{sec:Main}
	Since the lattices $L_e$, $L_o$ are finite all the trajectories of the system 
	\eqref{eq:f with r} are periodic.  The total increment in the 
	second coordinate of any periodic trajectory has to be proportional to 
	$q$. If the total increment of a trajectory is zero, we will call such a 
	trajectory bounded or periodic. Otherwise the trajectory will be called 
	escaping. 
\subsection{Existence of escaping trajectories.}
	\begin{thm} \label{thm: periodic for q even} Let 
	$\alpha=\frac{p}{q}$ 
	and 
	$y_0=\frac{a}{b}$ be two rational numbers satisfying the condition 
	$\lfloor 
	bq/2\rfloor = p 
	a\Mod{b}$. Then 
		\begin{enumerate}
			\item For $bq$ even, any orbit of the transformation 
			\eqref{eq:original map} starting at the level $y=y_0$ is 
			bounded.
			\item For $bq$ odd there exists a unique class of equivalent 
			trajectories of the system \eqref{eq:original map} starting at the 
			level $y=y_0$ and growing without bounds.
			\end{enumerate}
			\end{thm}
			
				\begin{figure*}[hbt]
					\begin{subfigure}{.35\textwidth}
						\includegraphics[width=\textwidth]{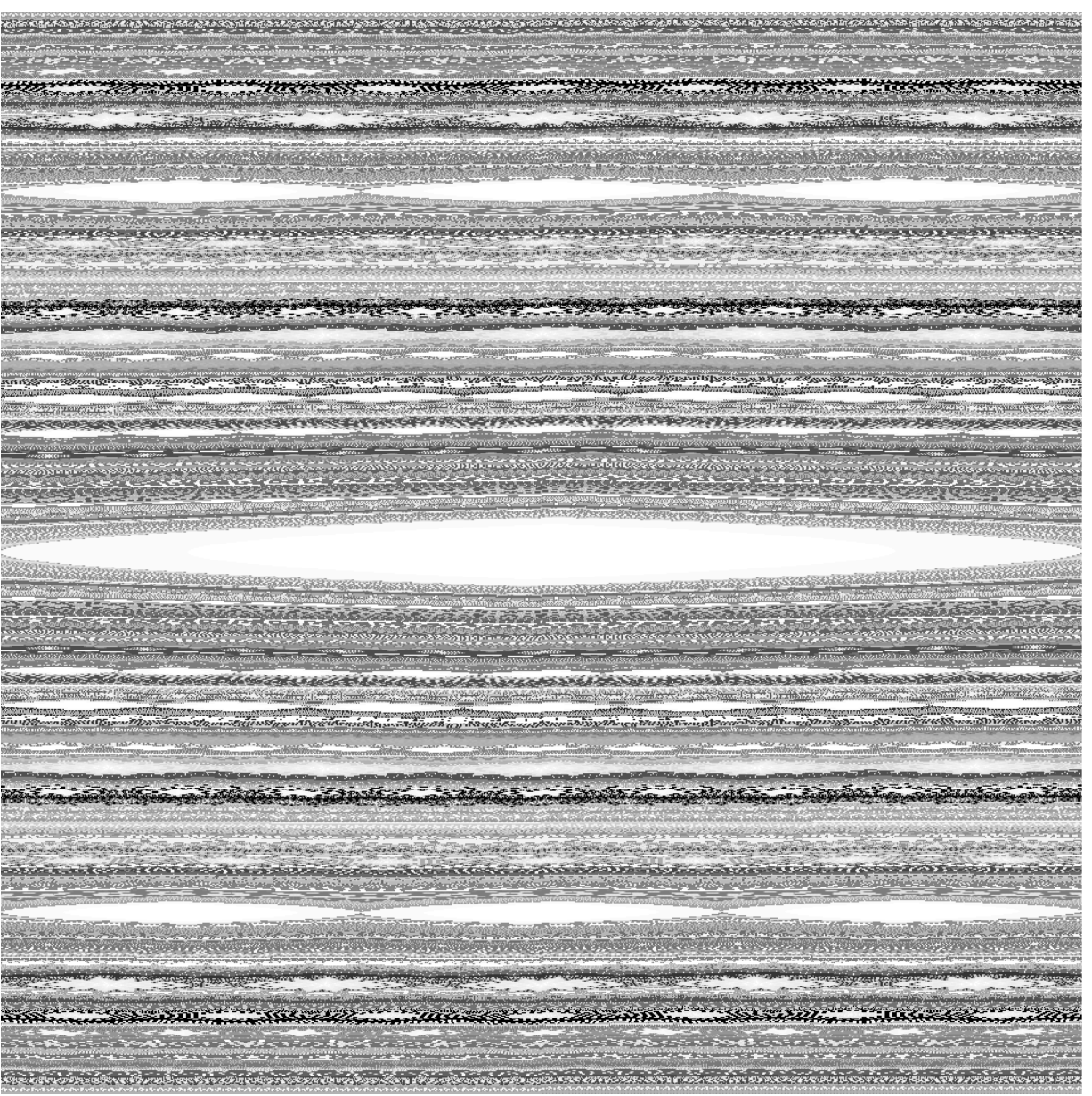}\caption{}
						\label{fig: islands}
					\end{subfigure}\qquad\qquad\qquad
					\centering \begin{subfigure}{.35\textwidth}
						\includegraphics[width=\textwidth]{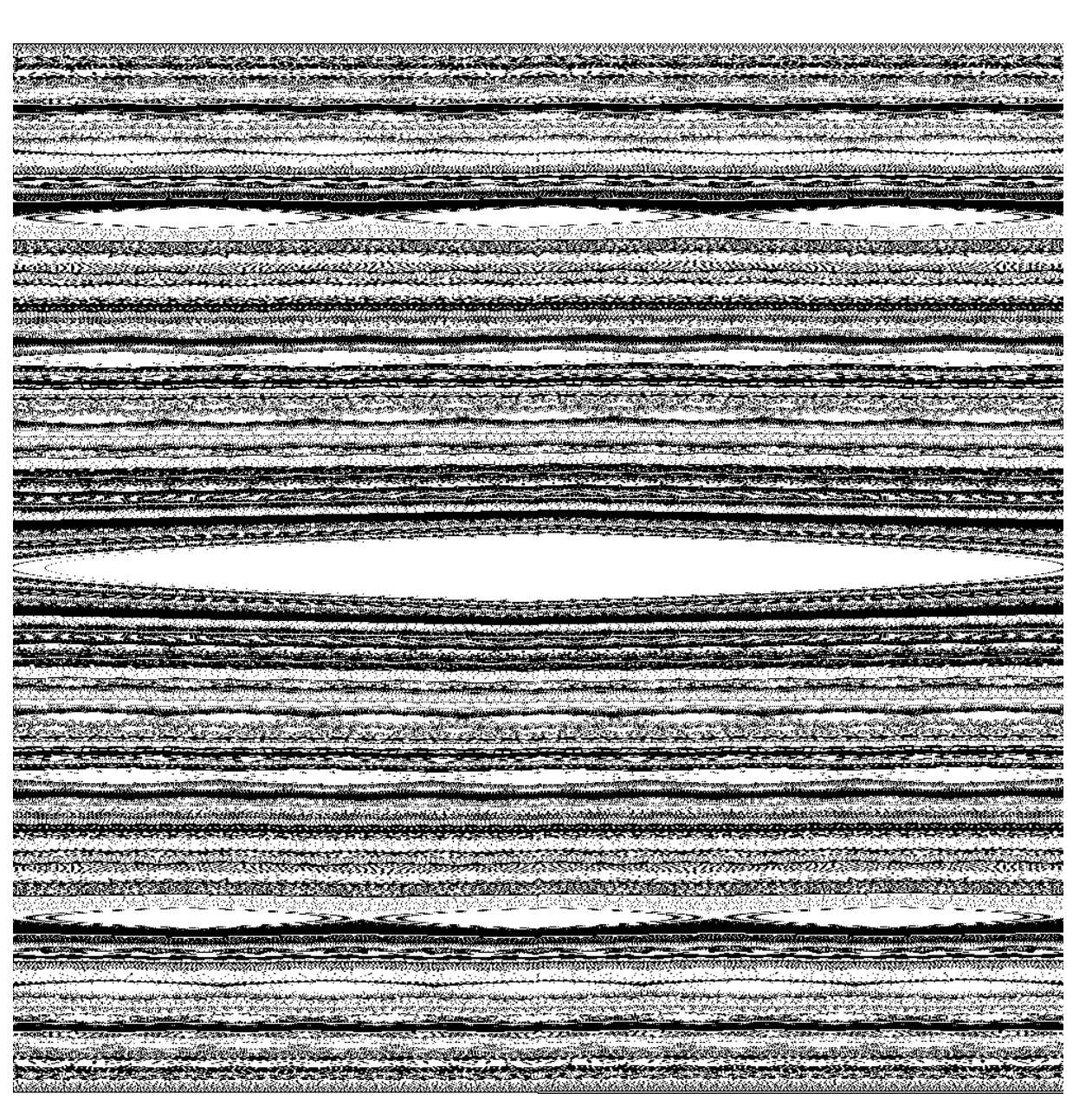}\caption{}
						\label{fig: wandering}
					\end{subfigure}
					
					\caption{(a): Phase space for 
						$q=992$ is filled 
						with the periodic orbits. Orbits are colored with the length 
						of the 
						period. Lighter points correspond to the shorter periods. 
						Largest 
						period equals $6168$. (b): Escaping orbit for $q=991$. 
						Length of the orbit 
						equals to $414639$.}
					\end{figure*}
\begin{proof} Thanks to the above discussion, every 
	unbounded trajectory of the system \eqref{eq:original map} corresponds 
	to the escaping trajectory of the system
	\eqref{eq:f with r}. Thus, one has to show that system \eqref{eq:f with r} 
	either does not have escaping trajectoies (for even $bq$) or 
	has exactly one escaping trajectory (for odd $bq$).
		Transformation \eqref{eq:f with r} can be considered as a continuous 
		transformation with respect to the second coordinate, since every 
		iteration gets an increment of $\pm 1$ in $j$. We will construct a 
		critical level $j=j_*$ such that no trajectory can cross it in the case of 
		even $bq$. As it will be clear from the construction, for the case of 
		odd $bq$ there is only one trajectory which can cross this level. 
		We will look for $j_*$ such that 
		\begin{equation}
		\label{eq:congruence}
		p(a+bj_*)=\lfloor bq/2 \rfloor\Mod{bq}.
		\end{equation} Note  that $p/q$ is irreducible and so 
		$\gcd(pb,bq)=b$. Since by the assumption of the theorem $\lfloor b 
		q/2\rfloor -pa$ is divisible by $b$ we conclude that the congruence 
		\eqref{eq:congruence} has exactly $b$ solutions in the form 
		$j_*+kq$, $k=0,1,\dots (b-1)$ (see \cite{number}). Thus all these 
		solutions correspond to the same equivalence class in $\Z_q$.   
		
		We will show that in the 
		case of even $bq$ 
		no trajectories may cross the level $j=j_*$ and for the odd $q$ there 
		is only 
		one such trajectory.
		Indeed,  assume for definiteness that for $(r, j_*-1)$ we have 
		$f(r,j_*-1)=(r',j_*)$. This means that $r'$ belongs to the right half of 
		the cylinder. But for the even $bq$ it follows that $bp j_*=bq/2$ and 
		so $r'$ 
		is shifted exactly by $bq/2$ thus the trajectory coming to the critical 
		level from below will go down at the next step. From lemma 
		\ref{lm:symmetry} it follows that neither trajectory can cross this level 
		from above. 
		
		For the case of odd $bq$ one gets $pj=(bq-1)/2$ and so only $r= 
		(bq-1)/2$ together with its image $r'= bq-1$ belong to the right half 
		of 
		the cylinder. Since there is a unique point at which trajectory may 
		pass 
		the level $j_*$ 
	such a trajectory necessarily has to be escaping (see Fig. 
	\ref{fig:bottle})	\end{proof}

	\begin{remark}
		In the case of integral initial condition $y_0=a$ one can set $b=1$ 
		and 
		so the congruence \eqref{eq:congruence} always has a solution. Thus 
		theorem \ref{thm: periodic for q even} states that for $\alpha=1/2k$ 
		all the 
		trajectories of the system \eqref{eq:original map} remain bounded 
		while for $\alpha=1/(2k+1)$ there is only one equivalence class of 
		unbounded trajectories.
		\end{remark}
		
		From numerical simulations the following statement is evident.
		
		\begin{conjecture}
			For every $\alpha=p/q$, there exists $y_0 = a/b$ such that there is 
			an escaping orbit of \eqref{eq:original map} starting at the level 
			$y=y_0$.
		\end{conjecture}
		
		For odd $q$ it follows from theorem \ref{thm: periodic for q even} 
		that $a=0$, $b= 1$ provides the desired result. When the  technical 
		conditions 
		of the Theorem \ref{thm: periodic for q even} is not satisfied, 
		congruence \eqref{eq:congruence} has 
		no solutions and we cannot construct the bottleneck level passing 
		which will assure that trajectory is escaping. However, one particular 
		case 
		seems to be tractable. From here on we fix $p$ to be equal to $1$.
		
\begin{thm}\label{thm:4k+2}
	For $q=4k+2$ there exists an escaping 
	trajectory of the transformation $f$, starting at the level $y_0=1/2$.
\end{thm}
Letting $a=1$, $b=2$ we will show that the orbit of the 
point 
$(r_0, j_0) 
= (4k-1,2k+1)$ is unbounded.
 From \eqref{eq:f with r} we get the following system
	\begin{equation*}
				\begin{cases}
				r' = r + 1+2j \Mod{2q} \\
				j' = j + \textrm{sgn}(1+2r'-2q)\Mod{q} 
				\end{cases}
				\end{equation*}
Next lemma provides us some control on the sub-lattice for 
which 
the orbit of $(r_0,y_0)$ should belong to.
				
				\begin{lemma}
					\label{lemma: mod 4 points}
					Let $q=4k+2$.  Denote by $(r_m,j_m)$ the 
					$m$-th 
					iterate of the point $(r_0,j_0)=(4k~-~1, 
					2k~+~1)$. 
					Then  
                     $r_m+m\Mod{2} = 3 \Mod{4}$.
				\end{lemma}
				\begin{proof} 
					At first we observe that the parity of the second coordinate of 
					the point always differs from the parity of $m$. Indeed $j_0$ is 
					odd and at each step the trajectory gets or looses $1$. 
					
					We proceed by induction. 
					First let us consider $m~=~0$. We have $j_0 = 
					2k+1=1\Mod{2}$ and $r_0 =  4k-1 = 3 \Mod{4}$. Then for 
					$m=1$ we get $r_1 = r_0+2j_0+1 
					= 3+2+1 = 2 \Mod{4}$. 
					
					Now assume that for some even $m$ the assumption of the 
					lemma holds true. Then since $j_m=1\Mod{2}$ it follows 
					$r_{m+1}=r_m+2j_m+1=2\Mod{4}$. 
					Finally if the assumption holds for some odd $m$ we get 
					$j_m~=~0\Mod{2}$ and therefore 
					$r_{m+1}~=~2~+~1~+~4~\Mod{4}~=~3~\Mod{4}$. 
				\end{proof}

				\begin{proof}[Proof of theorem \ref{thm:4k+2}]
					Consider the orbit of the point 
					$(r_0,j_0) 
					= (4k-1,2k+1)$.
					One can easily calculate that $j_1 = j_0+1$ and $j_2 = 
					j_1+1$, 
					that is, there are immediately two consecutive increases. It 
					then 
					suffices to show that there is no point at level $j = 
					2k+3$ 
					from which there are two consecutive decreases.
					
					To have a decrease to the 
					$j=2k+1$ level from $(r,2k+2)$ we need to have 
					$1+2r'-2q < 0$, i.e.,
					\begin{equation}
					\label{eq:inequality}
					r+q+3 
					\Mod{2q} 
					< q-\frac{1}{2} 
					\end{equation}
				%
					For $r<q-3$ we have $r+q+3 \Mod{2q} = r+q+3$, so for the 
					inequality (\ref{eq:inequality}) to hold we need $r < 
					-\frac{1}{2}-3$, 
					which is impossible. For $r \geq q-3$ we have $r+q+3 
					\Mod{2q} = 
					r-q+3$, so for (\ref{eq:inequality}) to hold we need 
					$r<2q-\frac{7}{2}$. 
					By lemma \ref{lemma: mod 4 points}, $r = 2 \Mod{4}$ for 
					any 
					point in our desired orbit at level $j = 2k+2$, and so 
					the 
					only $r$'s that are possibly in the orbit and result in a decrease 
					from this level are $r = 4k+2, 4k+6, \ldots, 8k-2$.
					
					In the $j = 2k+3$ level, to have a decrease to the $j = 
					2k+2$ level we need to have $r' 
					< 
					4k+\frac{3}{2}$. To have two consecutive decreases, this $r'$ 
					must be 
					one of the $r$'s we found in the previous paragraph. But the 
					smallest such $r$ is $4k+2$, so this cannot happen.
				\end{proof}
		
		For $q={4k}$, we searched 
		for $y_0=a/b$ that give an escaping orbit for some $(x_0,y_0) \in 
		L_e$. 
		We present the table of $a/b$ depending on $k$ with the 
		smallest 
		$b$.
		
		\begin{flushright}
			\begin{tabular}{l|*{8}{|c}r}
				$k$ & 1 & 2 & 3 & 4 & 5 & 6 & 7 & 8 \\
				\hline \hline $a$ & 1 & 4 & 4 & 1 & 26& 36 & 67 & 63  \\
				\hline $b$ & 3 & 13 & 11 & 45 & 57 & 103 & 144 & 205 \\ 
				\hline
			\end{tabular}
			
			\begin{tabular}{l|*{8}{|c}|r|}
				$k$&$\cdots\;\;\,9$ & 10 & 11 & 12 & 13 & 14    \\
				\hline \hline  $a$&$\cdots\;\, 77$ & 19 & 23 
				& 360 & 243 & 23   \\
				\hline $b$&$\cdots\,227$ & 337 & 
				223 & 1043 & 1264 & 505 \\ 
				\hline
			\end{tabular}
		\end{flushright}	
		
		One can see that the $b$ required increases rather quickly with $k$. It 
		also appears that one cannot simply narrow the search by taking 
		$a=1$. For example, for $k=3$ we searched for escaping orbits with 
		$y_0 = 1/b$ and found none for $b \leq 5000$.
		
		\subsection{Length of the escaping orbit.}
		Now we will investigate the growth rate of the escaping orbit. The 
		fastest possible rate fo the transformation \eqref{eq:original map} is 
		linear, i.e. the trajectory may gain as much as $O(N)$ in the second 
		coordinate after $N$ iterations. In fact, escaping trajectories grow 
		much slower. Since the phase space of the transformation \eqref{eq:f 
		with r} is finite and thus so are all the trajectories we will consider the 
		lengths of the trajectories instead of their growth rates. Let 
		$\alpha=1/q$ and $y_0=0$. Theorem \ref{thm: periodic for q even} 
		provides unique 
		escaping trajectory for each odd $q$. 
			\begin{definition}
			Define $\ell(q)$ as the unique odd-length period under $f$ on 
			$L_o$. 
		\end{definition}
		Quantity $\ell(q)$ describes the portion of the phase space 
		$\Z_q\times \Z_q$ swiped by the escaping trajectory. Thus linearly 
		growing trajectory should have $\ell(q)=O(q)$, since such a trajectory 
		should visit every level of the lattice $O(1)$ amount of times.  In fact, 
		numerical experiments show that the escaping trajectory has the 
		slowest 
		possible growth rate $\ell(q)=O(q^2)$ (see Fig. \ref{fig: fraction 
		of q^2})
		\begin{figure}[h]
				\centering
				\includegraphics[width=0.4\textwidth]{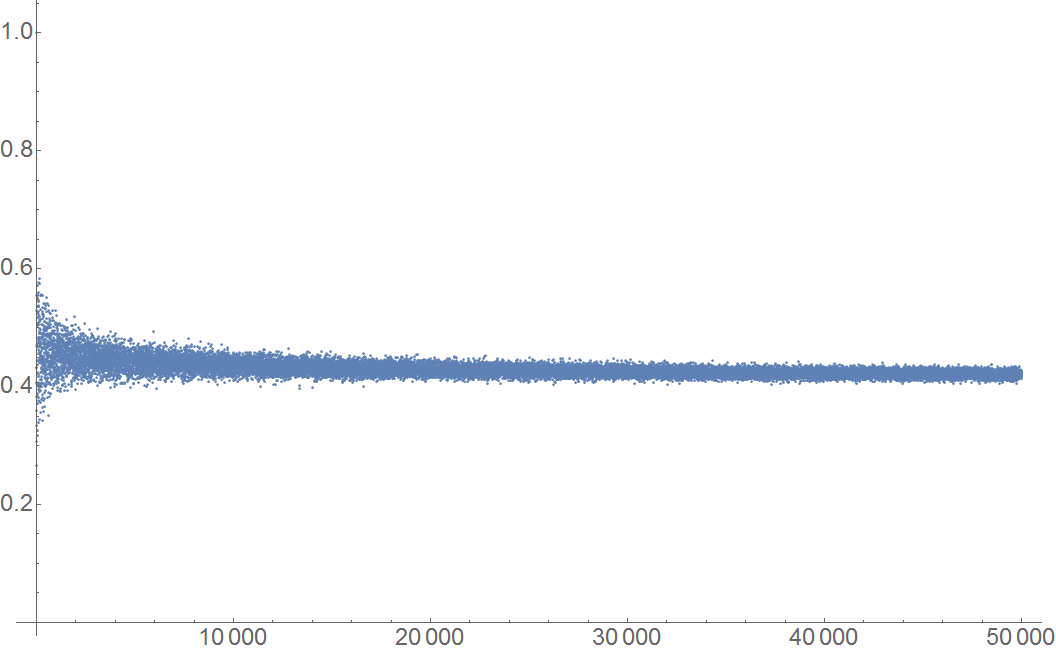}
				\caption{$\ell(q)/q^2$ for $q < 50~000$. The 
					average 
					value of $\ell(q)/q^2$ was found to be about 
					$0.43$.}
				\label{fig: fraction of q^2}
			\end{figure}
			\begin{conjecture}\label{con:asymptotic}
			Length of the escaping orbit $\ell(q)$ grows 
			as 
			$O(q^2)$.
		\end{conjecture}	
		
		At this moment we can provide much milder estimate.
			\begin{thm}
				\label{Theorem 1} Consider transformation 
				\eqref{eq:original map} with $\alpha=\frac 1 q$ with odd 
				$q$ and $y_0=0$. Let $\{x_n,y_n\}$ denote the unbounded 
				trajectory provided 
				by theorem \ref{thm: periodic for q even}. Then for any $n$ and 
				$n'$ such that $|n'-n|<q\log q$ it follows that $|y_{n'}-y_n|<q$.  
			\end{thm}
		
		We will prove this theorem providing an aprioiry bound on 
		the length 
		$\ell(q)$.
		The idea of the proof consists in the estimate of the time it takes from 
		the escaping trajectory to pass the levels near the bottleneck level 
		$j=(q+1)/2$. We will use two lemmas. First lemma states that 
		two consecutive vertical 
		increases near $j=(q+1)/2+m$ (for $m \geq 0$ reasonably small) 
		cause the resulting iterate to be less than about $k$ units to the right 
		of $x=1/2$ (a line of discontinuity for $f$). Second lemma uses the 
		information about the first coordinate of the trajectory to estimate the 
		time trajectory will spent on the prescribed level (see 
		Fig \ref{fig:bottle}).
		
		\begin{lemma}
			\label{Lemma 1}
			Let $q \geq 1$ be odd. Suppose $(r_0,j_0) \in L_o$, 
			$(r_1,j_1)=f(r_0,j_0)$, and $(r_2,j_2)=f(r_1,j_1)$ are such that 
			$j_0 = \frac{q+1}{2}+m-2$, $j_1 = j_0+1$, and $j_2 
			= 
			j_0+2$ for some $m \in \{1,2,\ldots,\frac{q-1}{2}\}$. 
			Then 
			$r_2\in[ \frac{q-1}{2} , \frac{q-1}{2}+~m~-~1]$.
		\end{lemma}
		\begin{proof}
			Let $j_0$, $j_1$, and $j_2$ be as in the statement of the lemma. 
			That $j_1>j_0$ means $r_1 > q/2$ (since we need 
			$\sgn(r_1-\frac{q}{2}) = 1$ in order for this to happen), and 
			in the same way $j_2>j_1 \implies r_2 > q/2$. Therefore 
			$\frac{q-1}{2} \leq r_i \leq q-1$ for $i \in \{1,2\}$. These 
			inequalities make sense, despite $\Z_q$ not being ordered, 
			because everything is in the interval $[0,q)$.
			
			Each value of $r_1$ satisfying these inequalities can be written as 
			$r_1 = \frac{q-1}{2}+n$ for some $n \in 
			\{0,1,\ldots,\frac{q-1}{2}\}$. We have 
			\[\begin{split} r_2 &= r_1 + j_1 \Mod{q} = 
			\left(\frac{q-1}{ 
			2}+n\right)+\\&+\left(\frac{q+1}{2}+m-1\right) 
			\Mod{q} 
			= n+m-1 \end{split}\]
			Using our definitions of $n$ and $m$, we obtain
			\[ \frac{q-1}{2} \leq r_2 = n+m-1 \leq \frac{q-1}{2}+m-1 \]
			as desired.
		\end{proof}
		
		\begin{figure}[hbt]
			\centering
			\includegraphics[width=0.4\textwidth]{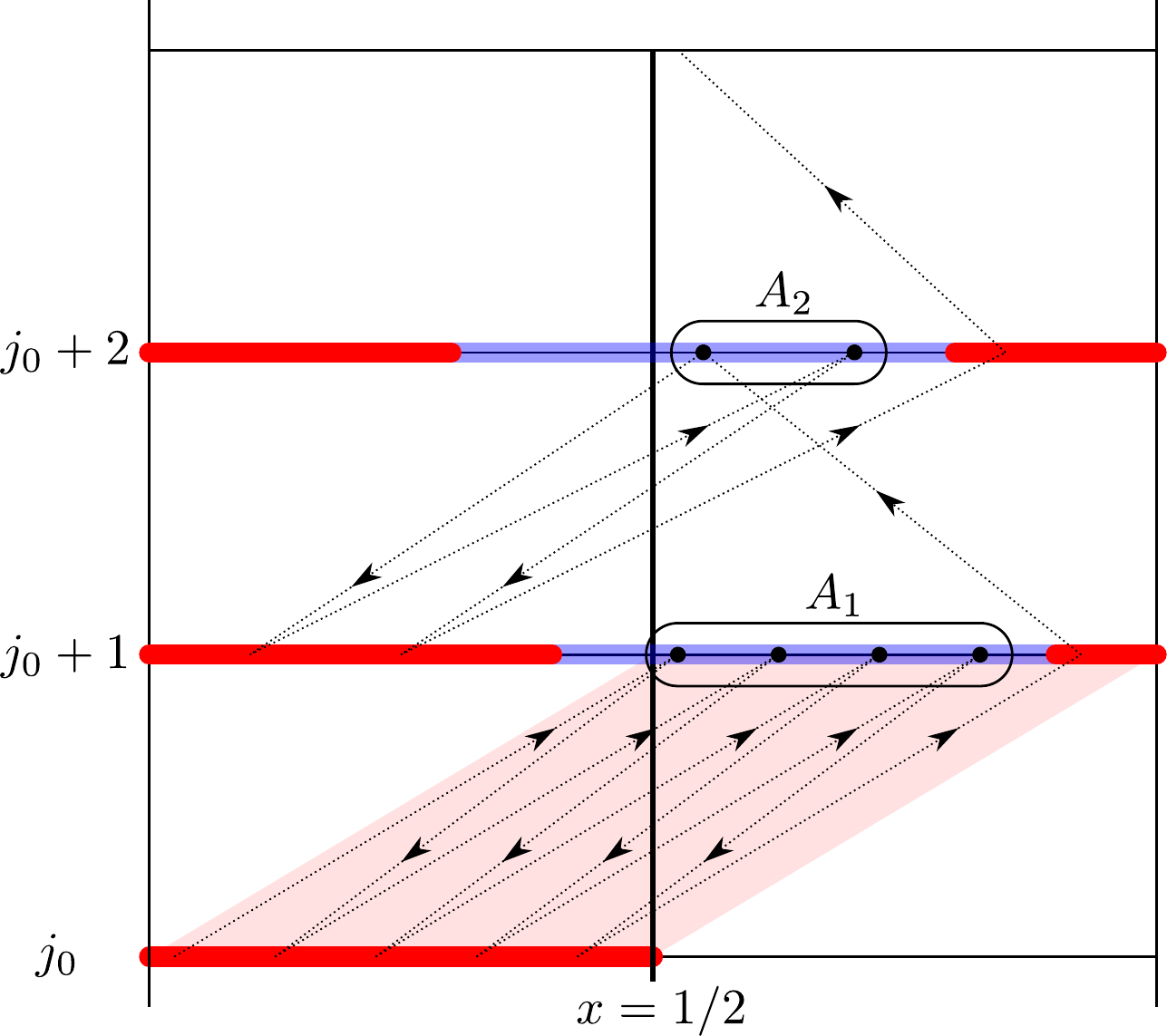}
			\caption{Escaping trajectory near the bottleneck 
			level.}\label{fig:bottle}
		\end{figure}
		
		The next lemma roughly states that if we start at a point within $k$ 
		units horizontally of $r=\frac{q+1}{2}$ (for $m>0$ reasonably small) 
		and at vertical level $j=\frac{q+1}{2}+m$, then the iterates of the 
		point 
		bounce at least about $q/2m$ times between $j=\frac{q+1}{2}+m$ 
		and 
		$j=\frac{q+1}{2}+m-1$.
		
		\begin{lemma}
			\label{Lemma 2}
			Let $q \geq 9$ be an odd integer. Let 
			$j_0=\frac{q+1}{2}+m$ 
			and 
			 $\frac{q+1}{2} \leq r_0 \leq \frac{q+1}{2}+m-1$ for $m
			\in 
			\{ 1,2, \ldots, \lfloor q/9 \rfloor \}$. Let $(r_0,j_0) \in L_o$ and 
			take $N_m$ to be the greatest integer such that $j_n \in \{ 
			\frac{q+1}{2}+m,~\frac{q+1}{2}+m-1 \}$ for all $n \leq N_m$. Then 
			$N_m \geq \left \lfloor \frac{q-1}{2m} \right \rfloor - 1$.
		\end{lemma}
			
		\begin{proof}
			Let $r_0 = \frac{q+1}{2}+s$, where $s$ is in $\{0,1,\ldots,m-1\}$. 
			We 
			have
			\[
			\begin{split}
			r_1 &= r_0+j_0 \Mod{q} = s+m+1<\frac{q}{2}, \\ j_1& = 
			j_0-1 
			= 
			\frac{q+1}{2}+m-1 
			\end{split}\] and
			\[\begin{split}
			r_2 &= r_1+j_1 \Mod{q} = \frac{q+1}{2}+s+2m > \frac{q}{2},\\ 
			j_2 
			&= j_1 + 1 = \frac{q+1}{2}+m.
			\end{split}\]
			Every two iterations of $f$, the value of $r$ increases by the amount
			$ \left(\frac{q+1}{2}+m\right) + \left(\frac{q+1}{2}+m-1\right) 
			\Mod{q} = 2m $
			until $r$ increases past $q$. Thus,
			\[
			(r_{2n-1},\,j_{2n-1}) = \left(s+(2n-1)m+1, 
			\,\frac{q+1}{2}+m-1\right) 
			\]
			and
			\[
			(r_{2n},\,j_{2n})= \left(\frac{q+1}{2}+s+2n m, \,\frac{q+1}{2}+m 
			\right)\]
			for all integers $n$ with $0\leq n \leq n^*$, where 
			$n^*$ is such that 
			$
			s+(2n^*-1)m+1 < \frac{q}{2}$ and 
			$\frac{q+1}{2}+s+2mn^* < q$.
		
			We claim that $n^* =\left \lfloor \dfrac{q-1}{4m} \right\rfloor -1$ 
			satisfies these inequalities. We have
			\[ s+(2n^*-1)m+1 < 
			m-1+\left(\frac{q-1}{2m}-1\right)m+1 
			=
			\frac{q-1}{2} \]
			And on the other hand
			\[ \frac{q+1}{2}+s+2mn^* \leq 
			\frac{q+1}{2}+m-1+\]\[+2\left(\frac{q-1}{4m}-1\right)m = q-m-1 
			< 
			q, \]
			as claimed. Note also that $n^* \geq 0$, since
			\[ \left\lfloor \frac{q-1}{4m} \right\rfloor-1 \geq \frac{q-1}{4m}-2 
			\geq \frac{9(q-1)}{4q}-2 = \]\[=\frac{1}{4}-\frac{9}{4q} \geq 
			\frac{1}{4}-\frac{1}{4} = 0 \]
			Therefore the total number of points $N_m$ with $j$ in 
			$\{\frac{q+1}{2}+m,\frac{q+1}{2}+m-1\}$ satisfies
			\[ N_m \geq 2n^*+1 = \left \lfloor \frac{q-1}{2m} \right 
			\rfloor-1 \]
		\end{proof}
		
		\begin{proof}[Proof of theorem \ref{Theorem 1}]
            From the proof of theorem  \ref{thm: periodic for q even} it follows 
            that the escaping orbit pass through the point $(x_0,(q-1)/2)$, 
            where  $x_0 = \frac{1}{4}+\frac{q-1}{2}$). There exist 
			positive integers $n_k$ for $k \in \{1,2,\ldots,\lfloor q/9 \rfloor\}$ 
			such that
		\[j_{n_k-2} = \frac{q+1}{2} + k - 2,\;
			j_{n_k-1} = \frac{q+1}{2} + k - 1 \]
			\[j_{n_k} = \frac{q+1}{2} + k 
			\]
			since the orbit must pass through at least one point at each height.
			
			By lemma \ref{Lemma 1}, $\frac{q+1}{2} \leq r_{n_k} \leq  
			\frac{q+1}{2}+k-1$. Define
			\[ A_k = \left\{ (r_{n_k+m},j_{n_k+m}): m=0,1,...,\left\lfloor 
			\frac{q-1}{2k} \right \rfloor -2  \right\} \]
			By lemma \ref{Lemma 2}, $|A_k| = \left \lfloor \frac{q-1}{2k} \right 
			\rfloor-1$. Since
			\begin{equation*}
			\frac{q}{2}<x_{n_k+m}<q, ~ y_{n_k+m}=\frac{q+1}{2}+k 
			~\textrm{for~} m \textrm{ even and}
			\end{equation*}
			\begin{equation*}
			0<x_{n_k+m}<\frac{q}{2}, ~ y_{n_k+m}=\frac{q+1}{2}+k-1 
			~\textrm{for~} m \textrm{ odd,}
			\end{equation*}
			the $A_k$ are disjoint. Therefore we have
			\[ \ell(q) \geq  
			\sum_{k=1}^{\lfloor 
			q/9 
			\rfloor} \left(\left \lfloor \frac{q-1}{2k} \right \rfloor -1 \right) = O 
			(q\log q) \]
			as desired.
		\end{proof}

	\section{Periodic orbits.}
	\label{sec:periodic}
	We conclude our discussion with the numerical investigation of the 
	distribution of 
	the periodic orbits.
	
		\begin{figure}[hbt]\centering
			\begin{subfigure}{.48\textwidth}
				\includegraphics[width=\textwidth]{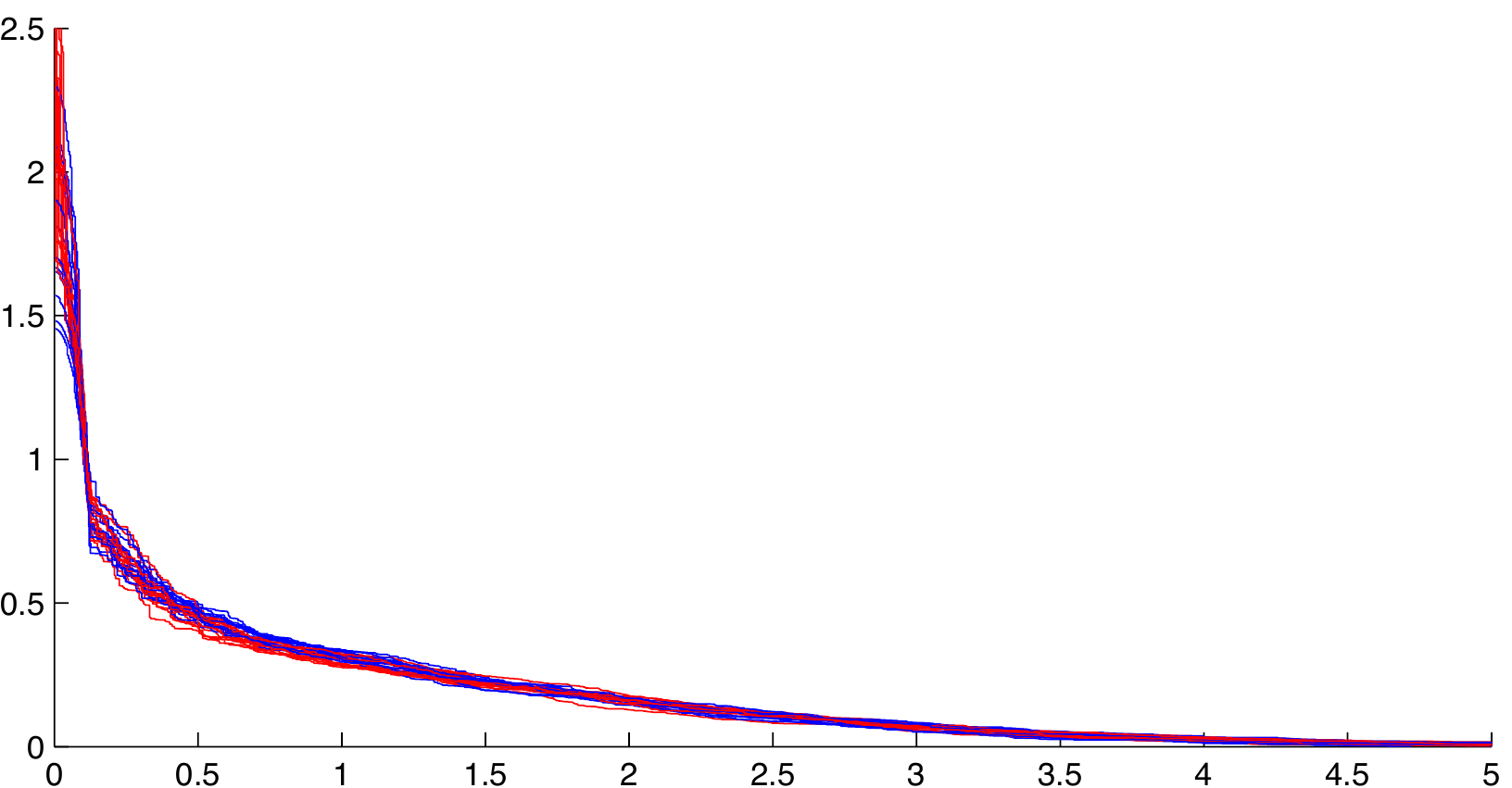}
				\caption{}\label{fig:young_4k+2}
			\end{subfigure}\par
			\centering \begin{subfigure}{.48\textwidth}
				\includegraphics[width=\textwidth]{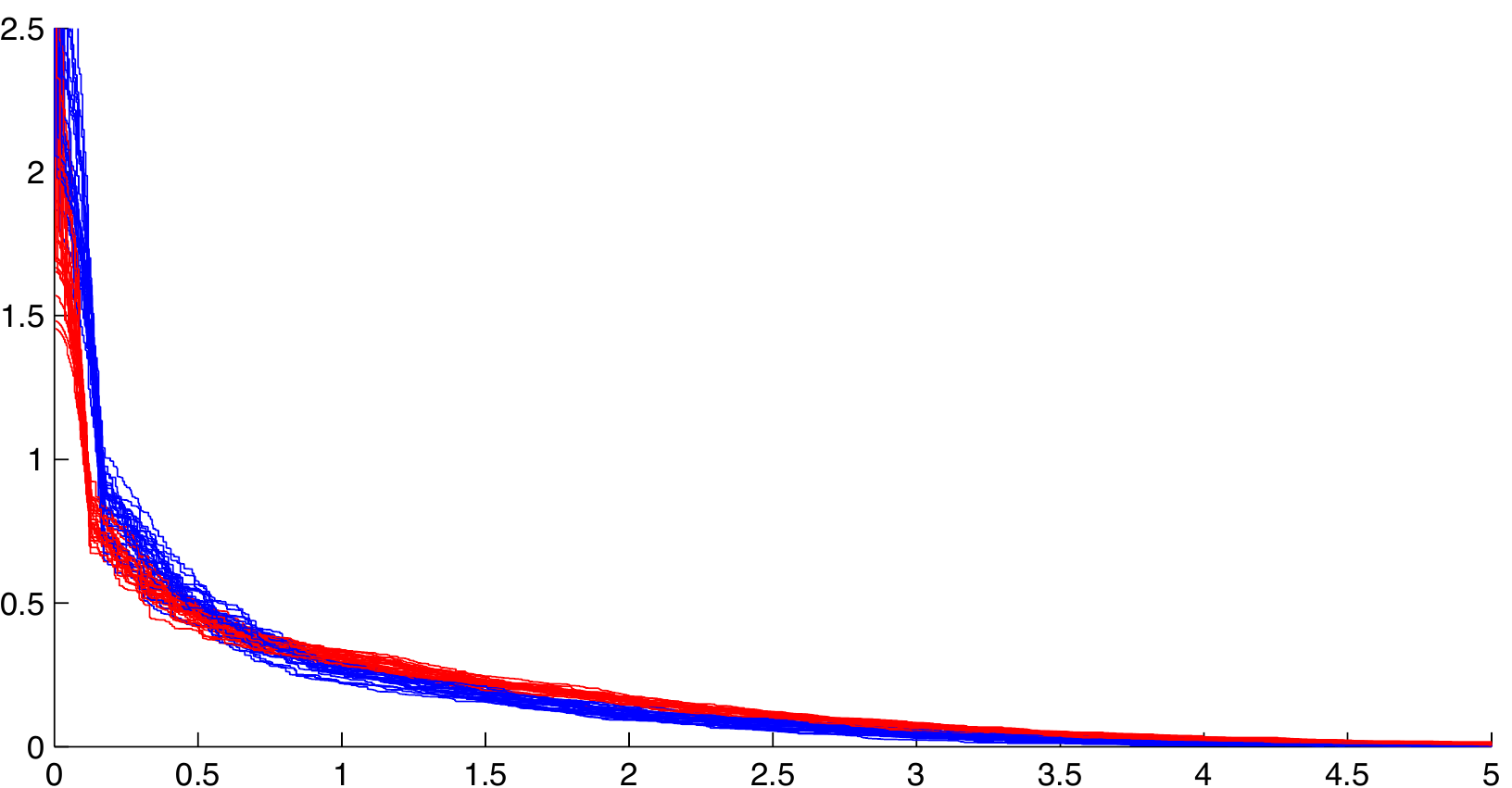}
				\caption{}\label{fig:young_odd}
			\end{subfigure}
			
			\caption{(a): Distribution
				of the preiodic orbits for even $q=950,\dots, 
				1000$. Case $q=4k$ is drawn in red. Blue color 
				corresponds to the case $q=4k+2$. (b): Distribution
				of the preiodic orbits  for $q=950,\dots, 
				1000$. Blue color corresponds to the case of odd $q$. Even $q$ 
				is 
				drawn in red.}
		\end{figure}
		 From theorem \ref{thm: periodic for q even} it follows 
	that the whole phase space of the system \eqref{eq:f with r} is divided 
	into 
	the set of periodic orbits of various periods. If $q$ is odd then there 
	exists a unique orbit of enormously large period which swipe almost a 
	half of the phase space. It turns out that all the other periods are 
	distributed in the range of $O(q)$. For even $q$ all the periods belong to 
	this range. What is spectacular that we observe some similarity in the 
	distribution of these periods for even and odd values of $q$. Collection 
	of the periodic orbits represents a partition of the number 
	$q^2$ 
	into the sum of the periods of the trajectories. We present here the 
	Young 
	diagrams for these partitions scaled by the factor of $q$ in both 
	directions. For the case of odd $q$ we present the diagram 
	corresponding to the partition of the set of bounded trajectories. It turns 
	out that the Young diagrams constructed for the cases of 
	even $q$ and for the bounded part of the phase space for the odd $q$ 
	are similar (see Fig. \ref{fig:young_odd}).

	\begin{conjecture}
		Maximum length of the bounded trajectories for the transformation 
		\eqref{eq:f 
			with 
			r} has the magnitude $O(q)$.
	\end{conjecture}

Looking at the portrait of the escaping trajectory (Fig. \ref{fig: wandering}) 
one can notice well-defined lacunae corresponding to the levels 
$j=q/(2n+1)$.  These lacunae represent the islands of stability around the 
corresponding periodic points for the transformation $f$.

\begin{figure}[hbt]
	\centering \includegraphics[width=0.4\textwidth]{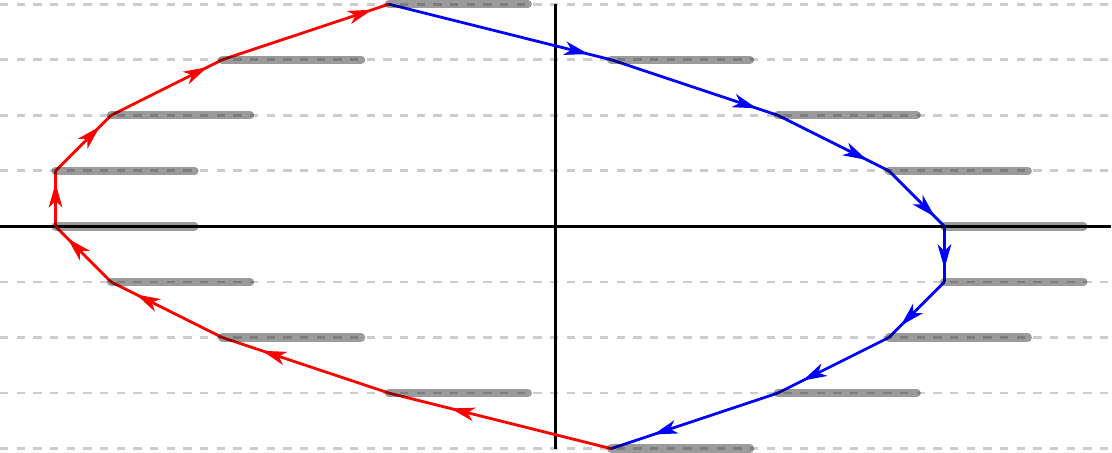}
	\caption{ Lacuna near the level $j=0$. Trajectories cover 
		$O(q)$ 
		distance in horizontal direction in $O(\sqrt{q})$ steps.} 
	\label{fig:lacunae}
\end{figure}

However, these islands  do not exhaust the whole phase space since the 
every island consists of $O(q)$ bounded trajectories while every such 
trajectory has period of order $O(q^{1/2})$ (see Fig \ref{fig:lacunae}). 
Nevertheless we observe that 
\begin{conjecture}
	Distributions of large periods of the periodic trajectories for even values 
	of $q$ coincide.
\end{conjecture}

On the other hand for small periods we have observed some differences. It 
turns out that  for 
$q=2\Mod{4}$  number of periodic orbits of small periods does not depend 
on $q$ while for $q=4k$ there are exactly $(2k-1)$ periodic orbits 
of period $4$. Indeed one can easily check by the direct computation that 
	trajectory of every point $(r,k)$, $r\in [2k,3k-1)$ is 
	$4$-periodic.
Combined with the lemma \ref{lm:symmetry}  this observation
provides $2k-2$ points of period $4$. Since the point $(0,0)$ is clearly 
$4$-periodic for any $q$, the total 
number of $4$-periodic orbits equals $2k-1$.
\bibliographystyle{plain}
\bibliography{pinball}
\end{document}